\documentclass[12pt]{article}
\usepackage{amssymb}
\usepackage{amsmath,amsthm}
\usepackage{lipsum}
\usepackage{graphicx}
\usepackage{hyperref}
\usepackage{enumerate}
\usepackage{multirow}
\usepackage{tikz}
\usepackage{mathrsfs}
\usepackage{caption}

\hypersetup{colorlinks=true, linkcolor=blue, citecolor=blue, urlcolor=blue}

\newtheorem{remark}{Remark}

\newtheorem{theorem}[remark]{Theorem}
\newtheorem{proposition}[remark]{Proposition}
\newtheorem{corollary}[remark]{Corollary}
\numberwithin{remark}{section}

\newcommand{\ter}{\operatorname{ter}}

\title{The $k$-metric dimension of a graph}

\author{Alejandro Estrada-Moreno$^{(1)}$, Juan A. Rodr\'{\i}guez-Vel\'{a}zquez$^{(1)}$,\\ and Ismael G. Yero$^{(2)}$
    \\
$^{(1)}${\small Departament d'Enginyeria Inform\`atica i Matem\`atiques,}\\
{\small Universitat Rovira i Virgili,}  {\small Av. Pa\"{\i}sos
Catalans 26, 43007 Tarragona, Spain.} \\{\small
alejandro.estrada\@@urv.cat, juanalberto.rodriguez\@@urv.cat}
\\
$^{(2)}${\small Departamento de Matem\'aticas, Escuela Polit\'ecnica Superior de Algeciras}\\
{\small Universidad de C\'adiz,} {\small
Av. Ram\'on Puyol s/n, 11202 Algeciras, Spain.} \\ {\small
ismael.gonzalez\@@uca.es}\\
}

\date{June 2, 2013  }

\begin{document}

\maketitle

\begin{abstract}
As a generalization of the concept of a metric basis,  this article introduces the notion of $k$-metric basis in graphs. Given a connected graph $G=(V,E)$, a set $S\subseteq V$ is said to be a $k$-metric generator for $G$ if the elements of any pair of different vertices of $G$ are distinguished by at least $k$ elements of $S$, {\em i.e.}, for any two different vertices $u,v\in V$, there exist at least $k$ vertices $w_1,w_2,\ldots,w_k\in S$ such that $d_G(u,w_i)\ne d_G(v,w_i)$ for every $i\in \{1,\ldots,k\}$. A metric generator of minimum cardinality is called a $k$-metric basis and its cardinality the $k$-metric dimension of $G$. A connected graph $G$ is  \emph{$k$-metric dimensional} if $k$ is the largest integer such that there exists a $k$-metric basis for $G$.   We give a necessary and sufficient condition for a graph to be $k$-metric dimensional and we obtain several results on the $k$-metric dimension.
\end{abstract}



\section{Introduction}

The problem of uniquely determining the location of an intruder in a network was the principal motivation of introducing the concept of metric dimension in graphs by Slater in \cite{Slater1975,Slater1988}, where the metric generators were called  locating sets. The concept of metric dimension of a graph was also introduced independently by Harary and Melter in \cite{Harary1976}, where metric generators were called resolving sets.

Nevertheless, the concept of a metric generator, in its primary version, has a weakness related with the possible uniqueness of the vertex identifying a pair of different vertices of the graph. Consider, for instance, some robots which are navigating, moving from node to node of a network. On a graph, however, there is neither the concept of direction nor that of visibility. We assume that robots have communication with a set of landmarks $S$ (a subset of nodes) which provide them the distance to the landmarks in order to facilitate the navigation. In this sense, one aim is that each robot is uniquely determined by the landmarks. Suppose that in a specific moment there are two robots $x,y$ whose positions are only distinguished by one landmark $s\in S$. If the communication between $x$ and $s$ is unexpectedly blocked, then the robot $x$ will get lost in the sense that it can assume that it has the position of $y$. So, for a more realistic settings it could be desirable to consider a set of landmarks where each pair of nodes is distinguished by at least two landmarks.

A natural solution regarding that weakness is the location of one landmark in every node of the graph. But, such a solution, would have a very high cost. Thus, the choice of a correct set of landmarks is convenient for a satisfiable performance of the navigation system. That is, in order to achieve a reasonable efficiency, it would be convenient to have a set of as few landmarks as possible, always having the guarantee that every object of the network will be properly distinguished.

From now on we consider a simple and connected graph $G=(V,E)$. It is said that a vertex $v\in V$ distinguishes two different vertices $x,y\in V$, if $d_G(v,x)\ne d_G(v,y)$, where $d_G(a,b)$ represents the length of a shortest $a-b$ path. A set $S\subseteq V$ is a \emph{metric generator} for $G$ if any pair of different vertices of $G$ is distinguished by some element of $S$. Such a name for $S$ raises from the concept of {\em generator} of metric spaces, that is, a set $S$ of points in the space with the property that every point of the space is uniquely determined by its ``distances'' from the elements of $S$. For our specific case, in a simple and connected graph $G=(V,E)$, we consider the metric $d_G:V\times V\rightarrow \mathbb{N}\cup \{0\}$, where $d_G(x,y)$ is defined as mentioned above and $\mathbb{N}$ is the set of positive integers. With this metric, $(V,d_G)$ is clearly a metric space. A metric generator of minimum cardinality is called a \emph{metric basis}, and its cardinality the \emph{metric dimension} of $G$, denoted by $\dim(G)$.

Other useful terminology to define the concept of a metric generator in graphs is given at next. Given an ordered set $S=\{s_{1}, s_{2}, \ldots, s_{d}\}\subset V(G)$, we refer to the $d$-vector (ordered $d$-tuple) $r(u|S)=$ $(d_{G}(u,s_{1}),$ $ d_{G}(u,s_{2}), \ldots, d_{G}(u,s_{d}))$ as the \emph{metric representation} of $u$ with respect to $S$. In this sense, $S$ is a metric generator for $G$ if and only if for every pair of different vertices $u,v$ of $G$, it follows $r(u|S)\neq r(v|S)$.

In order to avoid the weakness of metric basis described above, from now on we consider an extension of the concept of metric generators in the following way. Given a simple and connected graph $G=(V,E)$, a set $S\subseteq V$ is said to be a \emph{$k$-metric generator} for $G$ if and only if any pair of different vertices of $G$ is distinguished by at least $k$ elements of $S$, {\em i.e.}, for any pair of different vertices $u,v\in V$, there exist at least $k$ vertices $w_1,w_2,\ldots,w_k\in S$ such that
\begin{equation}\label{conditionDistinguish}
d_G(u,w_i)\ne d_G(v,w_i),\; \mbox{\rm for every}\; i\in \{1,\ldots,k\}.
\end{equation}
A $k$-metric generator of the minimum cardinality in $G$ will be  called a \emph{$k$-metric basis} and its cardinality the \emph{$k$-metric dimension} of $G$, which will be denoted by $\dim_{k}(G)$.

As an  example we take the cycle graph $C_{4}$ with vertex set $V=\{x_{1}, x_{2}, x_{3}, x_{4}\}$ and edge set $E=\{x_{i}x_{j}:j-i=1\pmod{2}\}$. We claim that $\dim_{2}(C_{4})=4$. That is, if we take the pair of vertices $x_{1}, x_{3}$, then they are distinguished only by themselves. So, $x_{1}, x_{3}$ must belong to every $2$-metric generator for $C_4$. Analogously, $x_{2},x_{4}$ also must belong to every $2$-metric generator for $C_4$. Other example is the graph $G$ in Figure \ref{figDim1}, for which $\dim_{2}(G)=4$. To see this, note that $v_{3}$ does not distinguish any pair of different vertices of $V(G)-\{v_{3}\}$ and for each pair $v_{i}, v_{3},$ $ 1\leq i\leq 5, i\not= 3$, there exist  two elements of $V(G)-\{v_{3}\}$ that distinguish them. Hence, $v_3$ does not belong to any $2$-metric basis for $G$.  To conclude  that  $V(G)-\{v_{3}\}$ must be a $2$-metric basis for $G$ we proceed as in the case of $C_{4}$.

\begin{figure}[!ht]
\centering
\begin{tikzpicture} [transform shape, inner sep = 1pt, outer sep = 0pt, minimum size = 20pt]
\node [draw=black, shape=circle, fill=white,text=black] (v1) at (0,2) {$v_1$};
\node [draw=black, shape=circle, fill=white,text=black] (v2) at (2,2) {$v_2$};
\node [draw=black, shape=circle, fill=white,text=black] (v3) at (3,1) {$v_3$};
\node [draw=black, shape=circle, fill=white,text=black] (v4) at (0,0) {$v_4$};
\node [draw=black, shape=circle, fill=white,text=black] (v5) at (2,0) {$v_5$};
\draw[black] (v1) -- (v2)  -- (v5) -- (v4) -- (v1);
\draw[black] (v1) -- (v3)  -- (v5);
\draw[black] (v2) -- (v3)  -- (v4);
\end{tikzpicture}
\caption{A graph $G$ where $V(G)-\{v_{3}\}$ is a $2$-metric basis for $G$.}
\label{figDim1}
\end{figure}
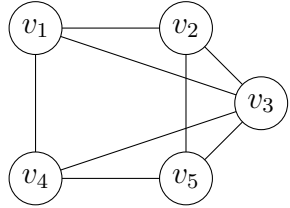

Note that every $k$-metric generator $S$ satisfies that $|S|\geq k$ and, if $k>1$, then $S$ is also a $(k-1)$-metric generator. Moreover,  $1$-metric generators
are the standard metric generators (resolving sets or locating sets as defined in \cite{Harary1976} or \cite{Slater1975}, respectively). Notice that if $k=1$, then the problem of checking if a set $S$ is a metric generator reduces to check condition (\ref{conditionDistinguish}) only for those vertices $u,v\in V- S$, as every vertex in $S$ is distinguished at least by itself. Also, if $k=2$, then condition (\ref{conditionDistinguish}) must be checked only for those pairs having at most one vertex in $S$, since two vertices of $S$ are distinguished at least by themselves. Nevertheless, if $k\ge 3$, then condition (\ref{conditionDistinguish}) must be checked for every pair of different vertices of the graph.

The literature about metric dimension in graphs shows several of its usefulness, for instance, applications to the navigation of robots in networks are discussed in \cite{Khuller1996} and applications to chemistry in \cite{Johnson1993,Johnson1998}, among others. This invariant was studied further in a number of other papers including \cite{Bailey2011a,Caceres2007,Chappell2008,Chartrand2000,Chartrand2000b,Fehr2006,Haynes2006,Okamoto2010,Saenpholphat2004,Tomescu2008,Yero2011,Yero2010}.  Several variations of metric generators including resolving dominating sets \cite{Brigham2003}, independent resolving sets \cite{Chartrand2000a}, local metric sets \cite{Okamoto2010}, and strong resolving sets \cite{Kuziak2013,Oellermann2007,Sebo2004}, etc. have been introduced and studied. It is therefore our goal to introduce this extension of metric generators in graphs as a possible future tool for other possibly more general variations of the applications described above.

We introduce now some other more necessary terminology for the article and the rest of necessary concepts will be introduced the first time they are mentioned in the work. 
We will use the notation $K_n$, $K_{r,s}$, $C_n$, $N_n$ and $P_n$ for complete graphs,  complete bipartite graphs, cycle graphs, empty graphs and path graphs, respectively. 
If two vertices $u,v$ are adjacent in $G=(V,E)$, then we write $u\sim v$ or we say that $uv\in E(G)$. Given  $x\in V(G)$ we define $N_{G}(x)$ to be the \emph{open neighbourhood} of $x$ in $G$. That is, $N_{G}(x)=\{y\in V(G):x\sim y\}$. The \emph{closed neighbourhood}, denoted by $N_{G}[x]$, equals $N_{G}(x)\cup \{x\}$. If there is no  ambiguity, we will simply write  $N(x)$ or $N[x]$. We also refer to the degree of $v$ as $\delta(v)=|N(v)|$. The minimum and maximum degrees of $G$ are denoted by $\delta(G)$ and $\Delta(G)$, respectively. For a non-empty set $S \subseteq V(G)$, and a vertex $v \in V(G)$, $N_S(v)$ denotes the set of neighbors that $v$ has in $S$, {\it i.e.}, $N_S(v) = S\cap N(v)$.

\section{$k$-metric dimensional graphs}

It is clear that it is not possible to find a $k$-metric generator in a connected graph $G$ for every integer $k$. 
That is, given a connected graph $G$, there exists an integer $t$ such that $G$ does not contain any $k$-metric generator for every $k>t$. According to that fact, a connected graph $G$ is said to be a \emph{$k$-metric dimensional graph}, if $k$ is the largest integer such that there exists a $k$-metric basis for $G$. Notice that, if $G$ is a $k$-metric dimensional graph, then for every positive integer $k'\le k$, $G$ has at least  a $k'$-metric basis. Since for every pair of different vertices $x,y$ of a graph $G$ we have that they are distinguished at least by themselves, it follows that the whole vertex set $V(G)$ is a $2$-metric generator for $G$ and, as a consequence it follows that every graph $G$ is $k$-metric dimensional for some $k\ge 2$. On the other hand, for any connected graph $G$ of order $n>2$ there exists at least one vertex $v\in V(G)$ such that $\delta(v)\ge 2$. Since $v$ does not distinguish any pair of different neighbours $x,y\in N_G(v)$,  there is no $n$-metric dimensional graph of order $n>2$.

\begin{remark}\label{remarkKMetric}
$\ $Let $G$ be a $k$-metric dimensional graph of order $n$. If $n\ge 3$, then $2\le k\le n-1$. Moreover, $G$ is $n$-metric dimensional if and only if $G\cong K_2$.
\end{remark}

Next we give a characterization of $k$-metric dimensional graphs. To do so, we need  some additional terminology. Given two different vertices $x,y\in V(G)$, we say that the set of \textit{distinctive vertices} of $x,y$ is $${\cal D}_G(x,y)=\{z\in V(G): d_{G}(x,z)\ne d_{G}(y,z)\}$$
and the set of \emph{non-trivial distinctive vertices} of  $x,y$ is $${\cal D}_G^*(x,y)={\cal D}_G(x,y)-\{x,y\}.$$

\begin{theorem}\label{theokmetric}
$\ $A connected graph  $G$ is $k$-metric dimensional  if and only if $k=\displaystyle\min_{x,y\in V(G), x\ne y}\vert {\cal D}_G(x,y)\vert .$
\end{theorem}

\begin{proof}
$\ $(Necessity) If $G$ is a $k$-metric dimensional graph, then for any $k$-metric basis $B$ and any pair of different vertices $x,y\in V(G)$, we have
$\vert B\cap {\cal D}_G(x,y)\vert \ge k.$
Thus, $k\le \displaystyle\min_{x,y\in V(G), x\ne y}\vert {\cal D}_G(x,y)\vert .$
Now, we suppose that
$k< \displaystyle\min_{x,y\in V(G), x\ne y}\vert {\cal D}_G(x,y)\vert.$ In such a case, for every $x',y'\in V(G)$ such that $\vert B\cap {\cal D}_G(x',y')\vert=k$, there exists a distinctive vertex $z_{x'y'}$ of $x',y'$ with $z_{x'y'}\in {\cal D}_G(x',y')-B$. Hence, the set $$B\cup \left(\displaystyle\bigcup_{x',y'\in V(G):\vert B\cap {\cal D}_G(x',y')\vert=k}\{z_{x'y'}\}\right)$$ is a $(k+1)$-metric generator for $G$, which is a contradiction. Therefore, $k=\displaystyle\min_{x,y\in V(G), x\ne y}\vert {\cal D}_G(x,y)\vert .$

(Sufficiency) Let $a,b\in V(G)$ such that $\displaystyle\min_{x,y\in V(G), x\ne y}\vert {\cal D}_G(x,y)\vert =\vert {\cal D}_G(a,b)\vert=k$. Since the set $$\bigcup_{x,y\in V(G)}{\cal D}_G(x,y)$$
is a $k$-metric generator for $G$ and  the pair $a,b$ is not distinguished by $k'>k$ vertices of $G$,  we conclude that $G$ is a $k$-metric dimensional graph.
\end{proof}

\subsection{On some families of $k$-metric dimensional graphs for some specific values of $k$}

The characterization proved in Theorem \ref{theokmetric} gives a result on general graphs. Thus, next we particularize this for some specific classes of graphs or we bound its possible value in terms of other parameters of the graph. To this end, we need the following concepts. Two vertices $x,y$ are called \emph{false twins}  if $N(x)=N(y)$ and $x,y$ are called \emph{true twins} if $N[x]=N[y]$. Two vertices $x,y$ are \emph{twins}  if they are false twins   or true twins. A vertex $x$ is said to be a \emph{twin}  if there exists a vertex $y\in V(G)-\{x\}$ such that $x$ and $y$ are twins in $G$. Notice that two vertices $x,y$ are twins if and only if ${\cal D}_G^*(x,y)=\emptyset.$

\begin{corollary}\label{remark2dimesional}
$\ $A connected graph $G$ of order $n\geq 2$ is $2$-metric dimensional if and only if $G$ has  twin vertices.
\end{corollary}

It is clear that $P_{2}$ and $P_{3}$ are $2$-metric dimensional. Now, a specific characterization for $2$-dimensional trees is obtained from Theorem \ref{theokmetric} (or from Corollary \ref{remark2dimesional}). A \emph{leaf} in a tree is a vertex of degree one, while a \emph{support vertex} is a vertex adjacent to a leaf.

\begin{corollary}\label{corolTree2}
$\ $A tree $T$ of order $n\geq 4$ is $2$-metric dimensional if and only if $T$ contains a support vertex which is adjacent to at least two leaves.
\end{corollary}

An example of a $2$-metric dimensional tree is the star graph $K_{1,n-1}$, whose $2$-metric dimension is $\dim_2(K_{1,n-1})=n-1$ (see Corollary \ref{corotree2}). On the other side, an example of a tree $T$ which is not $2$-metric dimensional is drawn in Figure \ref{figTreeD3}. Notice that $S=\{v_{1},v_{3},v_{5},v_{6},v_{7}\}$ is a $3$-metric basis of $T$. Moreover, $T$ is $3$-metric dimensional since $|{\cal D}_T(v_1,v_3)|=3.$

\begin{figure}[!ht]
\centering
\begin{tikzpicture}[transform shape, inner sep = 1pt, outer sep = 0pt, minimum size = 20pt]
\node [draw=black, shape=circle, fill=white,text=black] (v1) at (0,0) {$v_1$};
\node [draw=black, shape=circle, fill=white,text=black] (v2) at (1,0) {$v_2$};
\node [draw=black, shape=circle, fill=white,text=black] (v3) at (2,.5) {$v_3$};
\node [draw=black, shape=circle, fill=white,text=black] (v4) at (2,-.5) {$v_4$};
\node [draw=black, shape=circle, fill=white,text=black] (v5) at (3,.5) {$v_5$};
\node [draw=black, shape=circle, fill=white,text=black] (v6) at (3,-.5) {$v_6$};
\node [draw=black, shape=circle, fill=white,text=black] (v7) at (4,-.5) {$v_7$};
\draw[black] (v1) -- (v2)  -- (v4) -- (v6) -- (v7);
\draw[black] (v2) -- (v3)  -- (v5);
\end{tikzpicture}
\caption{$S=\{v_{1},v_{3},v_{5},v_{6},v_{7}\}$ is a $3$-metric basis of $T$.}
\label{figTreeD3}
\end{figure}
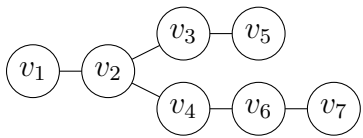

A {\em cut vertex} in a graph is a vertex whose removal increases the number of components of the graph and an {\em extreme vertex} is a vertex $v$ such that the subgraph induced by $N[v]$ is isomorphic to a complete graph. Also, a {\em block} is a maximal biconnected subgraph\footnote{A biconnected graph is a connected graph having no articulation vertices.} of the graph. Now, let $\mathfrak{F}$ be the family of sequences of connected graphs $G_1,G_2,\ldots,G_t$, $t\ge 2$, such that $G_1$ is a complete graph $K_{n_1}$, $n_1\ge 2$, and $G_i$, $i\ge 2$, is obtained recursively from $G_{i-1}$ by adding a complete graph $K_{n_i}$, $n_i\ge 2$, and identifying one vertex of $G_{i-1}$ with one vertex of $K_{n_i}$.

From this point we will say that a connected graph $G$ is a \emph{generalized tree}\footnote{In some works these graphs are called block graphs.} if and only if there exists a sequence $\{G_1,G_2,\ldots,G_t\}\in \mathfrak{F}$ such that $G_t=G$ for some $t\ge 2$. Notice that in these generalized trees every vertex is either, a cut vertex or an extreme vertex. Also, every complete graph used to obtain the generalized tree is a block of the graph. Note, that if every $K_{n_i}$ is isomorphic to $K_2$, then $G_t$ is a tree, justifying the terminology used. With these concepts we give the following consequence of Theorem \ref{theokmetric}, which is a generalization of Corollary \ref{corolTree2}.

\begin{corollary}\label{coro-gen-tree}
$\ $A generalized tree $G$ is $2$-metric dimensional if and only if $G$ contains at least two extreme vertices being adjacent to a common cut vertex.
\end{corollary}

The \emph{Cartesian product graph} $G\square H$, of two graphs $G=(V_{1},E_{1})$ and $H=(V_{2},E_{2})$, is the graph whose vertex set is $V(G\square H)=V_{1}\times V_{2}$ and any two distinct vertices $(x_{1},x_{2}),(y_{1},y_{2})\in V_{1}\times V_{2}$ are adjacent in $G\square H$ if and only if either:

\begin{enumerate}[(a)]
\item\label{cartesian1}$\ $ $x_{1}=y_{1}$ and $x_{2}\sim y_{2}$, or
\item\label{cartesian2}$\ $ $x_{1}\sim y_{1}$ and $x_{2}=y_{2}$.
\end{enumerate}

\begin{proposition}\label{corolGCartesian}
$\ $Let $G$ and $H$ be two connected graphs of order $n\ge 2$ and $n'\ge 3$, respectively. If $G\square H$ is $k$-metric dimensional, then $ k\geq 3$.
\end{proposition}

\begin{proof}
$\ $Notice that for any vertex $(a,b)\in V(G\square H)$, $N_{G\square H}((a,b))=(N_{G}(a)\times \{b\})\cup(\{a\}\times N_{H}(b))$. Now, for any two distinct vertices $(a,b),(c,d)\in V(G\square H)$ at least $a\ne c$ or $b\ne d$ and since $H$ is a connected graph of order greater than two, we have that at least $N_{H}(b)\ne \{d\}$ or $N_{H}(d)\ne \{b\}$. Thus, we obtain that $N_{G\square H}((a,b))\ne N_{G\square H}((c,d))$. Therefore, $G\square H$ does not contain any twins and, by Remark \ref{remarkKMetric} and Corollary \ref{remark2dimesional}, if $G\square H$ is $k$-metric dimensional, then $k\geq 3$.
\end{proof}

\begin{proposition}\label{propKClyce}
$\ $Let $C_{n}$ be a cycle graph of order $n$. If $n$ is odd, then $C_n$ is $(n-1)$-metric dimensional and if $n$ is even, then  $C_n$ is $(n-2)$-metric dimensional.
\end{proposition}

\begin{proof}
$\ $We consider two cases:
\begin{enumerate}[(1)]
\item $\ $ $n$ is odd. For any pair of different vertices $u,v\in V(C_{n})$ there exist only one vertex $w\in V(C_{n})$ such that $w$ does not distinguish $u$ and $v$.  Therefore, by Theorem \ref{theokmetric}, $C_n$ is $(n-1)$-metric dimensional.
\item $\ $ $n$ is even. In this case, $C_{n}$ is $2$-antipodal\footnote{The diameter of $G=(V,E)$ is defined as $D(G)=\max_{u,v\in V(G)}\{d_{G}(u,v)\}$. We say that $u$ and $v$ are antipodal vertices or mutually antipodal if $d_G(u,v)=D(G)$. We recall that $G=(V,E)$ is $2$-antipodal if for each vertex $x\in V$ there exists exactly one vertex $y\in V$ such that $d_G(x,y)=D(G)$.}. For any pair of vertices $u,v\in V(C_{n})$, such that $d(u,v)=2l$, we can take a vertex $x$ such that $d(u,x)=d(v,x)=l$. So,  ${\cal D}_G(u,v)=V(C_n)-\{x,y\}$, where $y$ is antipodal to $x$. On the other hand, if $d(u,v)$ is odd, then ${\cal D}_G(u,v)=V(C_n)$.  Therefore, by Theorem \ref{theokmetric}, the graph $C_n$ is $(n-2)$-metric dimensional.
\end{enumerate}
\end{proof}

Now, according to Remark \ref{remarkKMetric} we have that every graph of order $n$, different from $K_2$, is $k$-metric dimensional for some $k\le n-1$. Next we characterize those graphs being $(n-1)$-metric dimensional.

\begin{theorem}
$\ $A graph $G$ of order $n\ge 3$ is $(n-1)$-metric dimensional if and only if $G$ is a path or $G$ is an odd cycle.
\end{theorem}

\begin{proof}
$\ $Since $n\ge 3$, by Remark \ref{remarkKMetric}, $G$ is $k$-metric dimensional for some $k\in \{2,\ldots, n-1\}$. 
Now,  for any pair of different vertices $u,v\in V(P_{n})$ there exists at most one vertex $w\in V(P_{n})$ such that $w$ does not distinguish $u$ and $v$. Then  $P_n$ is $(n-1)$-metric dimensional. By Proposition \ref{propKClyce}, we have that if $G$ is an odd cycle, then $G$ is $(n-1)$-metric dimensional.

On the contrary, let $G$ be a $(n-1)$-metric dimensional graph. Hence, for every pair of different vertices $x,y\in V(G)$ there exists at most one vertex which does not distinguish $x,y$. Suppose $\Delta(G)>2$ and let $v\in V(G)$ such that $\{u_1,u_2,u_3\}\subset N(v)$. Figure \ref{figCases} shows all the possibilities for the links between these four vertices. Figures \ref{figCases} (a), \ref{figCases} (b) and \ref{figCases} (d) show that $v,u_1$ do not distinguish $u_2,u_3$. Figure \ref{figCases} (c) shows that $u_1,u_2$ do not distinguish $v,u_3$. Thus, from the cases above we deduce that there is a pair of different vertices  which is not distinguished by at least two other different vertices. Thus $G$ is not a $(n-1)$-metric dimensional graph, which is a contradiction. As a consequence, $\Delta(G)\le 2$ and we have that $G$ is either a path or a cycle graph. Finally,  by Proposition \ref{propKClyce}, we have that if $G$ is a cycle, then $G$ has odd order.
\end{proof}

\begin{figure}[h]
\centering
\begin{tikzpicture}[transform shape, inner sep = 1pt, outer sep = 0pt, minimum size = 20pt]
\node [draw=black, shape=circle, fill=white,text=black] (ua1) at (180:1.42) {$u_1$};
\node [draw=black, shape=circle, fill=white,text=black] (va) at (0,0) {$v$};
\node [draw=black, shape=circle, fill=white,text=black] (ua2) at (315:1.42) {$u_2$};
\node [draw=black, shape=circle, fill=white,text=black] (ua3) at (45:1.42) {$u_3$};
\draw[black] (ua1) -- (va)  -- (ua2);
\draw[black] (va)  -- (ua3);
\node at ([yshift=-1.8 cm]va) {(a)};

\node [draw=black, shape=circle, fill=white,text=black, xshift=4 cm] (ub1) at (180:1.42) {$u_1$};
\node [draw=black, shape=circle, fill=white,text=black, xshift=4 cm] (vb) at (0,0) {$v$};
\node [draw=black, shape=circle, fill=white,text=black, xshift=4 cm] (ub2) at (315:1.42) {$u_2$};
\node [draw=black, shape=circle, fill=white,text=black, xshift=4 cm] (ub3) at (45:1.42) {$u_3$};
\draw[black] (ub1) -- (vb)  -- (ub2) -- (ub3) -- (vb);
\node at ([yshift=-1.8 cm]vb) {(b)};

\node [draw=black, shape=circle, fill=white,text=black] (vc) at ([xshift=3 cm]ub2) {$v$};
\node [draw=black, shape=circle, fill=white,text=black] (uc1) at ([shift=({180:1.42})]vc) {$u_1$};
\node [draw=black, shape=circle, fill=white,text=black] (uc2) at ([shift=({0:1.42})]vc) {$u_2$};
\node [draw=black, shape=circle, fill=white,text=black] (uc3) at ([shift=({90:1.42})]vc) {$u_3$};
\draw[black] (uc1) -- (vc)  -- (uc2) -- (uc3) -- (uc1);
\draw[black] (vc) -- (uc3);
\node at ([yshift=-.8 cm]vc) {(c)};

\node [draw=black, shape=circle, fill=white,text=black] (vd) at ([xshift=8 cm]vb) {$v$};
\node [draw=black, shape=circle, fill=white,text=black] (ud1) at ([shift=({225:1.42})]vd) {$u_1$};
\node [draw=black, shape=circle, fill=white,text=black] (ud2) at ([shift=({315:1.42})]vd) {$u_2$};
\node [draw=black, shape=circle, fill=white,text=black] (ud3) at ([shift=({90:1.42})]vd) {$u_3$};
\draw[black] (ud1) -- (vd)  -- (ud2) -- (ud3) -- (ud1) -- (ud2);
\draw[black] (vd) -- (ud3);
\node at ([yshift=-1.8 cm]vd) {(d)};
\end{tikzpicture}
\caption{Possible cases for a vertex $v$ with three adjacent vertices $u_1,u_2,u_3$.}\label{figCases}
\end{figure}
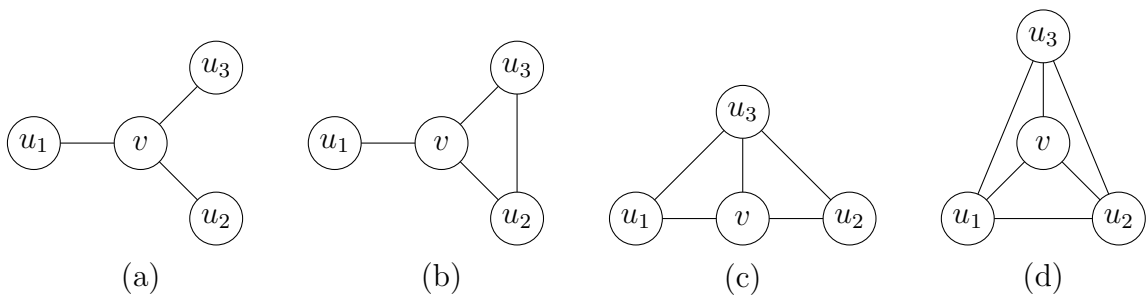

\subsection{Bounding the value $k$ for which a graph is $k$-metric dimensional}\label{SectionBoundK-dimensional}

In order to continue presenting our results, we need to introduce some definitions. A vertex of degree at least three in a graph $G$ will be called a \emph{major vertex} of $G$. Any end-vertex (a vertex of degree one) $u$ of $G$ is said to be a \emph{terminal vertex} of a major vertex $v$ of $G$ if $d_{G}(u,v)<d_{G}(u,w)$ for every other major vertex $w$ of $G$. The \emph{terminal degree} $\ter(v)$ of a major vertex $v$ is the number of terminal vertices of $v$. A major vertex $v$ of $G$ is an \emph{exterior major vertex} of $G$ if it has positive terminal degree. Let $\mathcal{M}(G)$ be the set of exterior major vertices of $G$ having terminal degree greater than one.

Given $w\in \mathcal{M}(G)$ and a terminal vertex $u_{j}$ of $w$, we denote by  $P(u_j,w)$ the shortest path that starts at $u_{j}$ and ends at $w$. Let $l(u_{j},w)$ be the length of $P(u_j,w)$. Now, given $w\in \mathcal{M}(G)$ and two terminal vertices $u_{j},u_{r}$ of $w$ we denote by  $P(u_{j},w,u_{r})$  the shortest path from $u_{j}$ to $u_{r}$ containing  $w$,
 and by  $\varsigma(u_{j},u_{r})$  the length of $P(u_{j},w,u_{r})$. Notice that, by definition of exterior major vertex, $P(u_{j},w,u_{r})$ is obtained by concatenating the paths $P(u_{j},w)$ and $P(u_{r},w)$, where $w$ is the only vertex of degree greater than two lying on these paths.

 Finally, given $w\in \mathcal{M}(G)$ and the set of terminal vertices $U=\{u_{1},u_{2},\ldots,u_{k}\}$ of $w$, for $j\not=r$ we define $\varsigma(w)=\displaystyle\min_{u_{j},u_{r}\in U}\{\varsigma(u_{j},u_{r})\}$ and  $l(w)=\displaystyle\min_{u_{j}\in U}\{l(u_{j},w)\}$.

\begin{figure}[!ht]
\centering
\begin{tikzpicture}[transform shape, inner sep = 1pt, outer sep = 0pt, minimum size = 20pt]
\node [draw=black, shape=circle, fill=white,text=black] (v1) at (0,0) {$v_1$};
\node [draw=black, shape=circle, fill=white,text=black] (v8) at (0,1.5) {$v_8$};
\node [draw=black, shape=circle, fill=white,text=black] (v12) at (1,3) {$v_{12}$};
\node [draw=black, shape=circle, fill=white,text=black] (v2) at ([shift=({1,0})]v1) {$v_2$};
\node [draw=black, shape=circle, fill=white,text=black] (v3) at ([shift=({1,0})]v2) {$v_3$};
\node [draw=black, shape=circle, fill=white,text=black] (v4) at ([shift=({1,0})]v3) {$v_4$};
\node [draw=black, shape=circle, fill=white,text=black] (v5) at ([shift=({1,0})]v4) {$v_5$};
\node [draw=black, shape=circle, fill=white,text=black] (v6) at ([shift=({1,0})]v5) {$v_6$};
\node [draw=black, shape=circle, fill=white,text=black] (v7) at ([shift=({1,0})]v6) {$v_7$};
\node [draw=black, shape=circle, fill=white,text=black] (v9) at ([shift=({1,0})]v8) {$v_9$};
\node [draw=black, shape=circle, fill=white,text=black] (v10) at ([shift=({4,0})]v9) {$v_{10}$};
\node [draw=black, shape=circle, fill=white,text=black] (v11) at ([shift=({1,0})]v10) {$v_{11}$};
\node [draw=black, shape=circle, fill=white,text=black] (v18) at ([shift=({1,0})]v11) {$v_{18}$};
\node [draw=black, shape=circle, fill=white,text=black] (v13) at ([shift=({1,0})]v12) {$v_{13}$};
\node [draw=black, shape=circle, fill=white,text=black] (v14) at ([shift=({1,0})]v13) {$v_{14}$};
\node [draw=black, shape=circle, fill=white,text=black] (v15) at ([shift=({1,0})]v14) {$v_{15}$};
\node [draw=black, shape=circle, fill=white,text=black] (v16) at ([shift=({1,0})]v15) {$v_{16}$};
\node [draw=black, shape=circle, fill=white,text=black] (v17) at ([shift=({1,0})]v16) {$v_{17}$};
\draw[black] (v1) -- (v2)  -- (v3) -- (v4) -- (v5) -- (v6) -- (v7);
\draw[black] (v8) -- (v9)  -- (v3) -- (v12);
\draw[black] (v3) -- (v13)  -- (v14) -- (v15) -- (v16) -- (v17);
\draw[black] (v11) -- (v15)  -- (v5) -- (v10);
\draw[black] (v4) -- (v14);
\draw[black] (v11) -- (v18);
\end{tikzpicture}
\caption{A graph $G$ where $\varsigma(G)=3$.}
\label{example-G-*}
\end{figure}
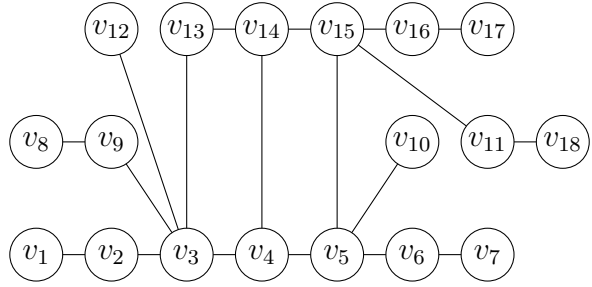

From the local parameters above we define the following global parameter $$\varsigma(G)=\min_{w\in \mathcal{M}(G)}\{\varsigma(w)\}.$$

An example which helps to understand the notation above is given in Figure \ref{example-G-*}.
%%%%%%%%%%%%%%  
In such a case we have $\mathcal{M}(G)=\{v_3,v_5,v_{15}\}$ and, for instance, $\{v_1,v_8,v_{12}\}$ are terminal vertices of $v_3$. So, $v_3$ has terminal degree three ($\ter(v_3)=3$) and it follows that
\begin{flalign*}
l(v_3)&=\min\{l(v_{12},v_3),l(v_8,v_3),l(v_1,v_3)\}=\min\{1,2,2\}=1,
\end{flalign*}
and
\begin{flalign*}
\varsigma(v_3)&=\displaystyle\min\{\varsigma(v_{12},v_1),\varsigma(v_{12},v_8),\varsigma(v_8,v_1)\}=\displaystyle\min\{3,3,4\}=3.
\end{flalign*}
Similarly, it is possible to observe that $\ter(v_5)=2$, $l(v_5)=1$, $\varsigma(v_5)=3$, $\ter(v_{15})=2$, $l(v_{15})=2$ and $\varsigma(v_{15})=4$. Therefore, $\varsigma(G)=3$.
%%%%%%%%%%%%%%%%%

 According to this notation we present the following result.

\begin{theorem}\label{coroKcota}
$\ $Let $G$ be a connected graph such that  $\mathcal{M}(G)\ne \emptyset$. If $G$ is $k$-metric dimensional, then $k\leq \varsigma(G)$.
\end{theorem}

\begin{proof}
$\ $We claim that there exists at least one pair of different vertices $x,y\in V(G)$ such that $\vert{\cal D}_G(x,y)\vert=\varsigma(G)$. To see this, let $w\in \mathcal{M}(G)$ and let $u_{1},u_{2}$ be two terminal vertices of $w$ such that $\varsigma(G)=\varsigma(w)=\varsigma(u_{1},u_{2})$. Let $u'_{1}$ and $u'_{2}$ be the vertices adjacent to $w$ in the shortest paths $P(u_{1},w)$ and $P(u_{2},w)$, respectively. Notice that it could happen $u'_{1}=u_{1}$ or $u'_{2}=u_{2}$. Since every vertex $v\not\in V\left(P(u_{1},w,u_{2})\right)-\{w\}$ satisfies that $d_{G}(u'_{1},v)=d_{G}(u'_{2},v)$, and the only distinctive vertices of $u'_{1},u'_{2}$ are those ones belonging to $P(u'_{1},u_{1})$ and $P(u'_{2},u_{2})$, we have that $\vert{\cal D}_G(u'_{1},u'_{2})\vert=\varsigma(G)$. Therefore, by Theorem \ref{theokmetric}, if $G$ is $k$-metric dimensional, then $k\leq \varsigma(G)$.
\end{proof}

The upper bound of Theorem \ref{coroKcota} is tight. For instance, it is achieved for every tree different from a path as it is proved further in Section \ref{sect-dim-trees}, where the $k$-metric dimension of trees is studied.

%%%%%%%%%%%%%%%%%

A {\em clique} in a graph $G$ is a set of vertices $S$ such that the subgraph induced by $S$, denoted by $\langle S\rangle$, is isomorphic to a complete graph. The maximum cardinality of a clique in a graph $G$ is the {\em clique number} and it is denoted by $\omega(G)$. We will say that $S$ is an $\omega(G)$-clique if $|S|=\omega(G)$.

\begin{theorem}\label{clique-versus-k}
$\ $Let $G$ be a graph of order $n$ different from a complete graph. If $G$ is $k$-metric dimensional, then $k\le n-\omega(G)+1$.
\end{theorem}

\begin{proof}
$\ $Let $S$ be an $\omega(G)$-clique. Since $G$ is not complete, there exists a vertex $v\notin S$ such that $N_S(v)\subsetneq S$. Let $u\in S$ with $v\not\sim u$. If $N_S(v)=S-\{u\}$, then $d(u,x)=d(v,x)=1$ for every $x\in S-\{u\}$. Thus, $\vert{\cal D}_G(u,v)\vert\le n-\omega(G)+1$.  On the other hand, if $N_S(v)\ne S-\{u\}$, then there exists $u'\in S-\{u\}$ such that $u'\not\sim v$. Thus, $d(u,v)=d(u',v)=2$ and for every $x\in S-\{u,u'\}$, $d(u,x)=d(u',x)=1$. So, $\vert{\cal D}_G(u,u')\vert\le n-\omega(G)+1$.   Therefore,  Theorem \ref{theokmetric} leads to  $k\le n-\omega(G)+1$.
\end{proof}

Examples where the previous bound is achieved are those connected graphs $G$ of order $n$ and clique number $\omega(G)=n-1$. In such a case, $n-\omega(G)+1=2$. Notice that in this case  there exists at least two twin vertices. Hence, by Corollary \ref{remark2dimesional} these graphs are $2$-metric dimensional.

The {\em girth} of a graph $G$ is the length of a shortest cycle in $G$.

\begin{theorem}\label{girth-versus-k}
$\ $Let $G$ be a graph of minimum degree $\delta(G)\ge 2$, maximum degree $\Delta(G)\ge 3$ and girth $\mathtt{g}(G)\ge 4$. If $G$ is $k$-metric dimensional, then
$$k\le n-1-(\Delta(G)-2)\sum_{i=0}^{\left\lfloor\frac{\mathtt{g}(G)}{2}\right\rfloor-2}(\delta(G)-1)^i.$$
\end{theorem}

\begin{proof}
$\ $Let $v\in V$ be a vertex of maximum degree in $G$. Since $\Delta(G)\ge 3$ and $\mathtt{g}(G)\ge 4$, there are at least three different vertices adjacent to $v$ and  $ N(v)$ is an independent set\footnote{An independent set or stable set is a set of vertices in a graph, no two of which are adjacent.}. Given $u_1,u_2\in N(v)$ and $i\in \{0,\ldots,\left\lfloor\frac{\mathtt{g}(G)}{2}\right\rfloor-2\}$ we define the following sets.
\begin{align*}
A_0&=N(v)-\{u_1,u_2\}.\\
A_1&=\bigcup_{x\in A_0}N(x)-\{v\}.\\
A_2&=\bigcup_{x\in A_1}N(x)-A_0.\\
&\ldots\\
A_{\left\lfloor\frac{\mathtt{g}(G)}{2}\right\rfloor-2}&=\bigcup_{x\in A_{\left\lfloor\frac{\mathtt{g}(G)}{2}\right\rfloor-3}}N(x)-A_{\left\lfloor\frac{\mathtt{g}(G)}{2}\right\rfloor-4}.
\end{align*}
Now, let $A=\{v\}\cup \left(\displaystyle\bigcup_{i=0}^{\left\lfloor\frac{\mathtt{g}(G)}{2}\right\rfloor-2}A_i\right)$. Since $\delta(G)\ge 2$, we have that $|A|\ge 1+(\Delta(G)-2)\displaystyle\sum_{i=0}^{\left\lfloor\frac{\mathtt{g}(G)}{2}\right\rfloor-2}(\delta(G)-1)^i$. Also, notice that for every vertex $x\in A$, $d(u_1,x)=d(u_2,x)$. Thus, $u_1,u_2$ can be only distinguished by themselves and at most $n-|A|-2$ other vertices. Therefore, $\vert{\cal D}_G(u_1,u_2)\vert\le  n-|A|$ and the result follows by Theorem \ref{theokmetric}.
\end{proof}

The bound of Theorem \ref{girth-versus-k} is sharp. For instance, it is attained for the graph  in Figure \ref{figUpperBound}. Since in this case $n=8$, $\delta(G)=2$, $\Delta(G)=3$ and $\mathtt{g}(G)=5$, we have that $k\le n-1-(\Delta(G)-2)\sum_{i=0}^{\left\lfloor\frac{\mathtt{g}(G)}{2}\right\rfloor-2}(\delta(G)-1)^i=6$. Table \ref{tableNearlyTwin} shows every pair of different vertices of this graph and their corresponding non-trivial distinctive vertices. Notice that  by Theorem \ref{theokmetric} the graph is $6$-metric dimensional.

\begin{figure}[!ht]
\centering
\begin{tikzpicture}[transform shape, inner sep = 1pt, outer sep = 0pt, minimum size = 20pt]
\node [draw=black, shape=circle, fill=white,text=black] (v1) at (0,0) {$v_1$};
\node [draw=black, shape=circle, fill=white,text=black] (v2) at ([shift=({1.41,0})]v1) {$v_2$};
\node [draw=black, shape=circle, fill=white,text=black] (v3) at ([shift=({1.41,0})]v2) {$v_3$};
\node [draw=black, shape=circle, fill=white,text=black] (v4) at ([shift=({45:1.41})]v3) {$v_4$};
\node [draw=black, shape=circle, fill=white,text=black] (v5) at ([shift=({135:1.41})]v4) {$v_5$};
\node [draw=black, shape=circle, fill=white,text=black] (v6) at ([shift=({-1.41,0})]v5) {$v_6$};
\node [draw=black, shape=circle, fill=white,text=black] (v7) at ([shift=({-1.41,0})]v6) {$v_7$};
\node [draw=black, shape=circle, fill=white,text=black] (v8) at ([shift=({225:1.41})]v7) {$v_8$};
\draw[black] (v1) -- (v2)  -- (v3) -- (v4) -- (v5) -- (v6) -- (v2);
\draw[black] (v6) -- (v7) -- (v8) -- (v1);
\end{tikzpicture}
\caption{A graph  that satisfies the equality in the upper bound of Theorem \ref{girth-versus-k}.}
\label{figUpperBound}
\end{figure}
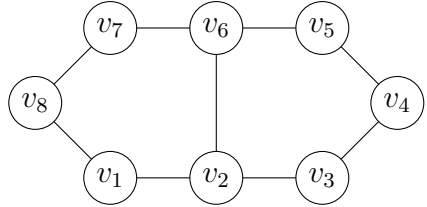

\begin{table}[h]

\centering
\begin{tabular}{|c|c|}
\hline
\mbox{$x,y$} & \mbox{${\cal D}_G^*(x,y)$}\\
\hline
 $v_1, v_3$ & $\{v_4, v_5, v_7, v_8\}$\\\hline
 $v_1, v_5$ & $\{v_2, v_4, v_6, v_8\}$\\\hline
 $v_1, v_6$ & $\{v_4, v_5, v_7, v_8\}$\\\hline
 $v_1, v_7$ & $\{v_2, v_3, v_5, v_6\}$\\\hline
 $v_1, v_8$ & $\{v_2, v_3, v_4, v_7\}$\\\hline
 $v_2, v_5$ & $\{v_1, v_3, v_4, v_8\}$\\\hline
 $v_2, v_6$ & $\{v_1, v_3, v_5, v_7\}$\\\hline
 $v_2, v_7$ & $\{v_1, v_3, v_4, v_8\}$\\\hline
 $v_3, v_4$ & $\{v_1, v_2, v_5, v_8\}$\\\hline
 $v_3, v_5$ & $\{v_1, v_2, v_6, v_7\}$\\\hline
 $v_3, v_6$ & $\{v_4, v_5, v_7, v_8\}$\\\hline
 $v_3, v_7$ & $\{v_2, v_4, v_6, v_8\}$\\\hline
 $v_4, v_5$ & $\{v_3, v_6, v_7, v_8\}$\\\hline
 $v_4, v_8$ & $\{v_1, v_3, v_5, v_7\}$\\\hline
 $v_5, v_7$ & $\{v_1, v_3, v_4, v_8\}$\\\hline
 $v_7, v_8$ & $\{v_1, v_4, v_5, v_6\}$\\
\hline
\end{tabular}
\hspace{0.5cm}
\begin{tabular}{|c|c|}
\hline
\mbox{$x,y$} & \mbox{${\cal D}_G^*(x,y)$}\\
\hline
 $v_1, v_2$ & $\{v_3, v_4, v_5, v_6, v_8\}$\\\hline
 $v_1, v_4$ & $\{v_2, v_3, v_5, v_7, v_8\}$\\\hline
 $v_2, v_3$ & $\{v_1, v_4, v_6, v_7, v_8\}$\\\hline
 $v_2, v_4$ & $\{v_1, v_5, v_6, v_7, v_8\}$\\\hline
 $v_2, v_8$ & $\{v_3, v_4, v_5, v_6, v_7\}$\\\hline
 $v_3, v_8$ & $\{v_1, v_2, v_4, v_5, v_7\}$\\\hline
 $v_4, v_6$ & $\{v_1, v_2, v_3, v_7, v_8\}$\\\hline
 $v_4, v_7$ & $\{v_1, v_3, v_5, v_6, v_8\}$\\\hline
 $v_5, v_6$ & $\{v_1, v_2, v_4, v_7, v_8\}$\\\hline
 $v_5, v_8$ & $\{v_1, v_3, v_4, v_6, v_7\}$\\\hline
 $v_6, v_7$ & $\{v_2, v_3, v_4, v_5, v_8\}$\\\hline
 $v_6, v_8$ & $\{v_1, v_2, v_3, v_4, v_5\}$\\\hline
  & \\\hline
  & \\\hline
  & \\\hline
  & \\
\hline
\end{tabular}
\caption{Pairs of vertices of the graph   in Figure \ref{figUpperBound} and their non-trivial distinctive vertices.}
\label{tableNearlyTwin}
\end{table}

\section{The $k$-metric dimension of graphs}

In this section we present some results that allow to compute the $k$-metric dimension of several families of graphs. We also give some tight bounds on the $k$-metric dimension of a graph.

\begin{theorem}[Monotony of the $k$-metric dimension]\label{theoBoundForKs}
$\ $Let $G$ be a $k$-metric dimensional graph and let $k_1,k_2$ be two integers. If $1\le k_1<k_2\le k$, then $\dim_{k_1}(G)<\dim_{k_2}(G)$.
\end{theorem}

\begin{proof}
$\ $Let $B$ be a $k$-metric basis of $G$. Let $x\in B$. Since all pairs of different vertices in $V(G)$ are distinguished by at least $k$ vertices of $B$, we have that $B-\{x\}$ is a $(k-1)$-metric generator for $G$ and, as a consequence, $\dim_{k-1}(G)\le \left|B-\{x\}\right|<|B|=\dim_{k}(G)$. Proceeding analogously, we obtain that $\dim_{k-1}(G)>\dim_{k-2}(G)$ and, by a finite repetition of the process we obtain the result.
\end{proof}

\begin{corollary}\label{firstConsequenceMonotony}
$\ $Let $G$ be a $k$-metric dimensional graph of order $n$.
\begin{enumerate}[{\rm (i)}]
\item $\ $For every $r\in\{1,\ldots, k\}$, $\dim_r(G)\ge \dim(G)+(r-1).$
\item $\ $For every $r\in\{1,\ldots, k-1\}$,  $\dim_r(G)<n$.
\item $\ $If $G\not\cong P_n$, then  for any $r\in \{1,\ldots,k\}$,
$\dim_{r}(G)\geq r+1.$
\end{enumerate}
\end{corollary}

\begin{proposition}\label{theoPath2}
$\ $Let $G$ be a connected graph of order $n\geq 2$. Then $\dim_{2}(G)=2$ if and only if $G\cong P_{n}$.
\end{proposition}

\begin{proof} $\;$ It was shown in 
 \cite{Chartrand2000}   that $\dim(G)=1$ if and only if $G\cong P_{n}$.

(Necessity) If $\dim_{2}(G)=2$, then by Corollary \ref{firstConsequenceMonotony} (i) we have that $\dim(G)=1$, {\it i.e.}, $$2=\dim_{2}(G)\ge \dim(G)+1\ge  2.$$
Hence, $G$ must be isomorphic to a path graph.

(Sufficiency) By Corollary \ref{firstConsequenceMonotony} (i) we have $\dim_2(P_n)\ge \dim(P_n)+1=2$ and, since the leaves of $P_n$ distinguish every pair of different vertices of $P_n$, we conclude that $\dim_2(P_n)=2$.
\end{proof}

%%%%%%%%%%%%%%%

Let ${\cal D}_k(G)$ be the set obtained as the union of the sets of distinctive vertices ${\cal D}_G(x,y)$ whenever $\vert{\cal D}_G(x,y)\vert=k$, {\it i.e.},
$${\cal D}_k(G)=\bigcup_{\vert {\cal D}_G(x,y)\vert=k}{\cal D}_G(x,y).$$

\begin{remark}\label{remTauk}
$\ $If $G$ is a $k$-metric dimensional graph, then $\dim_{k}(G)\geq \vert {\cal D}_k(G)\vert$.
\end{remark}

\begin{proof}
$\ $Since every pair of different  vertices $x,y$ is distinguished only by the elements of ${\cal D}_G(x,y)$, if
$\vert {\cal D}_G(x,y)\vert =k$, then for any $k$-metric basis $B$ we have
${\cal D}_G(x,y)\subseteq B$ and, as a consequence,
${\cal D}_k(G)\subseteq B$. Therefore, the result follows.
\end{proof}

The bound given in  Remark \ref{remTauk} is tight. For instance, in Proposition \ref{propTauk} we will show   that there exists a family of trees attaining this bound for every $k$. Other examples can be derived from the following result.

\begin{proposition}\label{theoDimkn}
$\ $Let $G$ be a $k$-metric dimensional graph of order $n$. Then $\dim_k(G)=n$ if and only if $V(G)={\cal D}_k(G)$.
\end{proposition}

\begin{proof}
$\ $Suppose that $V(G)={\cal D}_k(G)$.  Now, since every $k$-metric dimensional graph $G$ satisfies that $\dim_k(G)\le n$, by Remark \ref{remTauk} we obtain that $\dim_k(G)=n$.

On the contrary, let $\dim_{k}(G)=n$. Note that for every
$a,b\in  V(G)$, we have $\vert {\cal D}_G(a,b)\vert \ge k$. If there exists at least one vertex $x\in V(G)$ such that $x\notin {\cal D}_k(G)$, then for every $a,b\in  V(G)$, we have $\vert {\cal D}_G(a,b)-\{x\}\vert \ge k$ and, as a consequence,  $V(G)-\{x\}$ is a  $k$-metric generator for $G$, which is a contradiction. Therefore, $V(G)={\cal D}_k(G)$.
\end{proof}

\begin{corollary}\label{remarkDim2n}
$\ $Let $G$ be a connected graph of order $n\geq 2$. Then $\dim_2(G)=n$ if and only if every vertex is a twin.
\end{corollary}

We will show other examples of graphs that satisfy Proposition \ref{theoDimkn} for $k\ge 3$. To this end, we recall that the \emph{join graph} $G+H$ of the graphs $G=(V_{1},E_{1})$ and $H=(V_{2},E_{2})$ is the graph with vertex set $V(G+H)=V_{1}\cup V_{2}$ and edge set $E(G+H)=E_{1}\cup E_{2}\cup \{uv\,:\,u\in V_{1},v\in V_{2}\}$. We give now some examples of graphs satisfying the assumptions of Proposition \ref{theoDimkn}. Let $W_{1,n}=C_n+K_1$ be the \emph{wheel graph} and $F_{1,n}=P_n+K_1$ be the \emph{fan graph}. The vertex of $K_1$ is called the central vertex of the wheel or the fan, respectively. Since $V(F_{1,4})={\cal D}_3(F_{1,4})$ and $V(W_{1,5})={\cal D}_4(W_{1,5})$, by Proposition \ref{theoDimkn} we have that $\dim_3(F_{1,4})=5$ and $\dim_4(W_{1,5})=6$,  respectively.

Given two non-trivial graphs $G$ and $H$, it holds that any pair of twin vertices $x,y\in V(G)$ or $x,y\in V(H)$ are also twin vertices in $G+H$. As a direct consequence of Corollary \ref{remarkDim2n}, the next result holds.

\begin{remark}\label{coroJoinGG}
$\ $Let $G$ and $H$ be two nontrivial graphs of order $n_{1}$ and $n_{2}$, respectively. If all the vertices of $G$ and $H$ are twin vertices, then $G+H$ is $2$-metric dimensional and $$\dim_{2}(G+H)=n_{1}+n_{2}.$$
\end{remark}

Note that in Remark \ref{coroJoinGG}, the graphs $G$ and $H$ could be non connected. Moreover,  $G$ and $H$ could be nontrivial empty graphs. For instance, $N_{r}+N_{s}$, where $N_r$, $N_s$, $r,s>1$, are empty graphs, is the complete bipartite graph $K_{r,s}$ which satisfies that $\dim_{2}(K_{r,s})=r+s$.

\subsection{Bounding the $k$-metric dimension of graphs}

We begin this subsection with a necessary definition of the \textit{twin  equivalence relation} ${\cal R}$ on $V(G)$ as follows:
$$x {\cal R} y \longleftrightarrow  N_G[x]=N_G[y] \; \; \mbox{\rm or } \; N_G(x)=N_G(y).$$

We have three possibilities for each twin equivalence class $U$:

\begin{enumerate}[(a)]
\item $\ $ $U$ is singleton, or
\item $\ $ $N_G(x)=N_G(y)$, for any $x,y\in U$ (and case (a) does not apply), or
\item $\ $ $N_G[x]=N_G[y]$, for any $x,y\in U$ (and case (a) does not apply).
\end{enumerate}

We will  refer to the type (c) classes as the \textit{true twin  equivalence classes} \textit{i.e.,} $U$ is a true twin  equivalence class if and only if $U$ is not singleton and $N_G[x]=N_G[y]$, for any $x,y\in U$.

Let us see three different examples where every vertex is a twin. An example of a graph where every equivalence class is a true twin equivalence class is $K_r+(K_s\cup K_t)$, $r,s,t\ge 2$. In this case, there are three equivalence
classes composed by $r,s$ and $t$  true twin vertices, respectively. As an example where no class is composed by true twin vertices we take the complete bipartite graph $K_{r,s}$, $r,s\ge 2$. Finally, the graph $K_r+N_s$, $r,s\ge 2$, has two equivalence classes and one of them is composed by $r$ true twin vertices. On the other hand, $K_1+(K_r\cup N_s)$, $r,s\ge 2$, is an example where one class is singleton, one class is composed by true twin vertices and the other one is composed by false twin vertices.

In general, we can state the following result.

\begin{remark}
$\ $Let $G$ be a connected graph  and let  $U_1,U_2,\ldots,U_t$ be the  non-singleton twin equivalence classes of $G$. Then
$$\dim_2(G)\ge \sum_{i=1}^t|U_i|.$$
\end{remark}

\begin{proof}
$\ $Since for  two different vertices $x,y\in V(G)$  we have that $ {\cal D}_2(x,y)=\{x,y\}$  if and only if there exists an equivalence class $U_i$ such that $x,y\in U_i$,  we deduce $${\cal D}_2(G)= \bigcup_{i=1}^t U_i.$$ Therefore,  by Remark \ref{remTauk} we conclude the proof.
\end{proof}

Notice that the result above leads to Corollary \ref{remarkDim2n}, so this bound is tight. Now  we consider the connected graph $G$ of order $r+s$ obtained from a null graph $N_r$ of order $r\ge 2$ and a path $P_s$ of order $s\ge 1$ by connecting every vertex of $N_r$ to a given leaf of $P_s$. In this case, there are $s$ singleton classes and one class, say $U_1$, of cardinality $r$. By the previous result we have $\dim_2(G)\ge \vert U_1\vert=r$ and, since $U_1$ is a 2-metric generator for $G$, we conclude  that $\dim_2(G)=r.$

We recall that the {\em strong product graph} $G\boxtimes H$ of two graphs $G=(V_{1},E_{1})$ and $H=(V_{2},E_{2})$ is the graph with vertex set $V\left(G\boxtimes H\right)=V_{1}\times V_{2}$, where two distinct vertices $(x_{1},x_{2}),(y_{1},y_{2})\in V_{1}\times V_{2}$ are adjacent in $G\boxtimes H$ if and only if one of the following holds.
\begin{itemize}
\item $\ $ $x_{1}=y_{1}$ and $x_{2}\sim y_{2}$, or
\item $\ $ $x_{1}\sim y_{1}$ and $x_{2}=y_{2}$, or
\item $\ $ $x_{1}\sim y_{1}$ and $x_{2}\sim y_{2}$.
\end{itemize}

\begin{theorem}\label{Generalizacompleto-por-G}
$\ $Let $G$ and $H$ be two nontrivial connected graphs of order $n$ and $n'$, respectively. Let $U_1,U_2,\ldots,U_{t}$ be the true twin equivalence classes of $G$. Then $$\dim_2(G\boxtimes H)\ge n'\sum_{i=1}^t|U_i|.$$
Moreover, if every vertex of $G$ is a true twin, then $$\dim_2(G\boxtimes H)= n n'.$$
\end{theorem}

\begin{proof}
$\ $For any two vertices $a,c\in U_i$ and $b\in V(H)$,
\begin{align*}N_{G\boxtimes H}[(a,b)]&=N_G[a]\times N_H[b]\\
&= N_G[c]\times N_H[b]\\
&=N_{G\boxtimes H}[(c,b)].
\end{align*}
Thus, $(a,b)$ and $(c,b)$ are true twin vertices.
Hence, $${\cal D}_2(G\boxtimes H)\supseteq \bigcup_{i=1}^t U_i\times V(H).$$ Therefore, by Remark \ref{remTauk} we conclude $\dim_2(G\boxtimes H)\ge n'\displaystyle\sum_{i=1}^t|U_i|.$

Finally, if every vertex of $G$ is a true twin, then $ \displaystyle\bigcup_{i=1}^t U_i=V(G)$ and, as a consequence, we obtain $\dim_2(G\boxtimes H)= nn'.$
\end{proof}

We now present a lower bound for the $k$-metric dimension of a $k'$-metric dimensional graph $G$ with $k'\ge k$. To this end, we require the use of the following function for any exterior major vertex $w\in V(G)$ having terminal degree greater than one, {\it i.e.}, $w\in \mathcal{M}(G)$. Notice that this function uses the concepts already defined in Section \ref{SectionBoundK-dimensional}.
Given an integer $r\le k'$,

\[
I_r(w)=\left\{ \begin{array}{ll}
\left(\ter(w)-1\right)\left(r-l(w)\right)+l(w), & \mbox{if } l(w)\le\lfloor\frac{r}{2}\rfloor,\\
& \\
\left(\ter(w)-1\right)\lceil\frac{r}{2}\rceil+\lfloor\frac{r}{2}\rfloor, & \mbox{otherwise.}
\end{array}
\right.
\]

In Figure \ref{example-G-*} we give an example of a graph $G$, which helps to clarify the notation above. 
Since every graph is at least $2$-metric dimensional, we can consider the integer $r=2$ and we have the following.
\begin{itemize}
\item $\ $Since $l(v_3)=1\le \left\lfloor\frac{r}{2}\right\rfloor$, it follows that $I_r(v_3)=\left(\ter(v_3)-1\right)\left(r-l(v_3)\right)+l(v_3)=(3-1)(2-1)+1=3$.
\item $\ $Since $l(v_5)=1\le \left\lfloor\frac{r}{2}\right\rfloor$, it follows that $I_r(v_5)=\left(\ter(v_5)-1\right)\left(r-l(v_5)\right)+l(v_5)=(2-1)(2-1)+1=2$.
\item $\ $Since $l(v_{15})=2>\left\lfloor\frac{r}{2}\right\rfloor$, it follows that $I_r(v_{15})=\left(\ter(v_{15})-1\right)\left\lceil\frac{r}{2}\right\rceil+\left\lfloor\frac{r}{2}\right\rfloor=(2-1)\left\lceil\frac{2}{2}\right\rceil+\left\lfloor\frac{2}{2}\right\rfloor=2$.
\end{itemize}
Therefore, according to the result below, $\dim_2(G)\geq 3+2+2=7$.

\begin{theorem}\label{theoMuk}
$\ $If $G$ is a $k$-metric dimensional graph such that $|\mathcal{M}(G)|\ge 1$, then for every $r\in \{1,\ldots,k\}$,
$$\dim_{r}(G)\geq \sum_{w\in \mathcal{M}(G)}I_{r}(w).$$
\end{theorem}

\begin{proof}
$\ $Let $S$ be an $r$-metric basis of $G$. Let $w\in \mathcal{M}(G)$ and let $u_{i},u_{s}$ be two different terminal vertices of $w$. Let $u'_{i},u'_{s}$ be the vertices adjacent to $w$ in the paths $P(u_{i},w)$ and $P(u_{s},w)$, respectively. Notice that  ${\cal D}_G(u'_{i},u'_{s})=$ $V\left(P(u_{i},w,u_{s})\right)-\{w\}$ and, as a consequence, it follows that $\left|S\cap \left(V\left(P(u_{i},w,u_{s})\right)-\{w\}\right)\right|\ge r$. Now, if $\ter(w)=2$, then we have $$\left|S\cap \left(V\left(P(u_{i},w,u_{s})\right)-\{w\}\right)\right|\ge r=I_{r}(w).$$
Now, we assume $\ter(w)>2$. Let $W$ be the set of terminal vertices of $w$, and let $u_{j}'$ be the vertex adjacent to $w$ in the path $P(u_{j},w)$ for every $u_{j}\in W$. Let $U(w)=\displaystyle\bigcup_{u_{j}\in W}V\left(P(u_{j},w)\right)-\{w\}$ and let $x=\displaystyle\min_{u_{j}\in W}\lbrace\left|S\cap V\left(P(u_{j},w)\right)\right|\rbrace$. Since $S$ is an $r$-metric generator of minimum cardinality (it is an $r$-metric basis of $G$), it is satisfied that $0\le x\le \min\lbrace l(w),\lfloor\frac{r}{2}\rfloor\rbrace$. Let $u_{\alpha}$ be a terminal vertex such that $\left|S\cap \left(V\left(P(u_{\alpha},w)\right)-\{w\}\right)\right|=x$. Since for every terminal vertex $u_{\beta}\in W-\{u_{\alpha}\}$ we have that $|S\cap {\cal D}_G(u_{\beta}',u_{\alpha}')|\ge r$, it follows that $\left|S\cap \left(V\left(P(u_{\beta},w)\right)-\{w\}\right)\right|\ge r-x$. Thus,
\begin{align*}
|S\cap U(w)|&=\left|S\cap \left(V\left(P(u_{\alpha},w)\right)-\{w\}\right)\right|+\\
&+\displaystyle\sum_{\beta=1,\beta\ne \alpha}^{\ter(w)}\left|S\cap \left(V\left(P(u_{\beta},w)\right)-\{w\}\right)\right|\\
&\ge \left(\ter(w)-1\right)(r-x)+x.
\end{align*}
Now, if $x=0$, then $|S\cap U(w)|\ge \left(\ter(w)-1\right)r>I_r(w)$. On the contrary, if $x>0$, then the function $f(x)=\left(\ter(w)-1\right)(r-x)+x$ is decreasing with respect to $x$. So, the minimum value of $f$ is achieved in the highest possible value of $x$. Thus, $|S\cap U(w)|\ge I_r(w)$. Since $\displaystyle\bigcap_{w\in \mathcal{M}(G)}U(w)=\emptyset$, it follows that
$$\displaystyle \dim_{r}(G)\geq \sum_{w\in \mathcal{M}(G)}|S\cap U(w)|\ge \sum_{w\in \mathcal{M}(G)}I_{r}(w).$$
\end{proof}

Now, in order to give some consequences of the bound above we will use some notation defined in Section \ref{SectionBoundK-dimensional}  to introduce the following parameter.

$$\mu(G)=\sum_{v\in{\cal M}(G)}\ter(v).$$
Notice that for $k=1$ Theorem \ref{theoMuk} leads to the bound on the metric dimension of a graph, established by Chartrand \textit{et al.} in \cite{Chartrand2000}. In such a case, $I_1(w)=\ter(w)-1$ for all $w\in \mathcal{M}(G)$ and thus, $$\dim(G)\geq \sum_{w\in \mathcal{M}(G)}\left(\ter(w)-1\right)=\mu(G)-|\mathcal{M}(G)|.$$
Next we give the particular cases of Theorem \ref{theoMuk} for $r=2$ and $r=3$.

\begin{corollary}\label{coroMu2}
$\ $If $G$ is a connected graph, then $$\dim_{2}(G)\geq \mu(G).$$
\end{corollary}

\begin{proof}
$\ $If $\mathcal{M}(G)=\emptyset$, then $\mu(G)=0$ and the result is direct. Suppose that $\mathcal{M}(G)\ne\emptyset$. Since  $I_2(w)=\ter(w)$ for all $w\in \mathcal{M}(G)$, we deduce that  $$\dim_{2}(G)\displaystyle\geq \sum_{w\in \mathcal{M}(G)}\ter(w)=\mu(G).$$
\end{proof}

\begin{corollary}\label{coroMu3}
$\ $If $G$ is $k$-metric dimensional for some $k\ge 3$, then $$\dim_{3}(G)\geq 2\mu(G)-|\mathcal{M}(G)|.$$
\end{corollary}

\begin{proof}
$\ $If $\mathcal{M}(G)=\emptyset$, then the result is direct. Suppose that $\mathcal{M}(G)\ne\emptyset$. Since $I_3(w)=2 \ter(w)-1$ for all $w\in \mathcal{M}(G)$, we obtain that  $$\dim_{3}(G)\geq \sum_{w\in \mathcal{M}(G)}\left(2 \ter(w)-1\right)=2\mu(G)-|\mathcal{M}(G)|.$$
\end{proof}

In next section we give some results on trees which show that the bounds proved in Theorem \ref{theoMuk} and Corollaries \ref{coroMu2} and \ref{coroMu3} are tight. Specifically those results are Theorem \ref{theoTreeDimR} and Corollaries \ref{corotree2} and \ref{corotree3}, respectively.

\section{The particular case of trees}\label{sect-dim-trees}

To study the $k$-metric dimension of a tree it is of course necessary to know first the value $k$ for which a given tree is $k$-metric dimensional. That is what we do next. In this sense, from  now on we need the terminology and notation already described in Section \ref{SectionBoundK-dimensional} and also the following one. Given an exterior major vertex $v$ in a tree $T$ and the set of its terminal vertices $v_1,\ldots,v_{\alpha}$, the subgraph induced by the set $\displaystyle\bigcup_{i=1}^{\alpha} V(P(v,v_i))$ is called a \emph{branch} of $T$ at $v$ (a $v$-branch for short).

\begin{theorem}\label{theoTreeDimK}
$\ $If $T$ is a $k$-metric dimensional tree different from a path, then $k=\varsigma(T)$.
\end{theorem}

\begin{proof}
$\ $Since $T$ is not a path,  $\mathcal{M}(T)\ne \emptyset$. 
%%%%%%%%%%%
Let  $w\in \mathcal{M}(T)$ and let $u_{1},u_{2}$ be two terminal vertices of $w$ such that $\varsigma(T)=\varsigma(w)=\varsigma(u_{1},u_{2})$. Notice that, for instance, the two neighbours of $w$ belonging to the paths $P(w,u_{1})$ and $P(w,u_{2})$, say $u'_{1}$ and $u'_{2}$ satisfy $|{\cal D}_T(u'_{1},u'_{2})|=\varsigma(T)$.
%%%%%

It only remains to prove that for every $x,y\in V(T)$ it holds that $|{\cal D}_T(x,y)|\ge\varsigma(T)$.
%%%%%
Let $w\in \mathcal{M}(T)$ and let  $T_{w}=(V_{w},E_{w})$ be the $w$-branch. Also we consider the set of vertices $V'=V(T)-\bigcup_{w\in \mathcal{M}(T) }V_{w}$. 
 Note that $|V_w|\ge \varsigma(T)+1$ for every $w \in \mathcal{M}(T)$. With this fact in mind, we consider three cases.

{\bf Case 1:} $x\in V_w$ and $y\in V_{w'}$ for some $w,w'\in \mathcal{M}(T)$, $w\ne w'$. In this case $x,y$ are distinguished  by $w$ or by $w'$. Now, if $w$ distinguishes the pair $x,y$, then at most one element of $V_w$ does not distinguish $x,y$ (see Figure \ref{counterexample}). So, $x$ and $y$ are distinguished by at least $|V_w|-1$ vertices of $T$ or by at least $|V_{w'}|-1$ vertices of $T$.\\

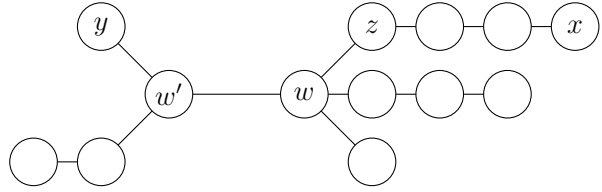
\begin{figure}[!ht]
\centering
\begin{tikzpicture}[scale=.9, transform shape, inner sep = 1pt, outer sep = 0pt, minimum size = 20pt]
\node [draw, shape=circle, fill=white] (wj) at (0,0) {$w'$};
\node [draw, shape=circle, fill=white] (y) at (-1,1) {$y$};
\node [draw, shape=circle, fill=white] (u1) at (-1,-1) {};
\node [draw, shape=circle, fill=white] (u2) at (-2,-1) {};
\draw (wj) -- (y);
\draw (wj) -- (u1);
\draw (u1) -- (u2);
\node [draw, shape=circle, fill=white] (wi) at (2,0) {$w$};
\node [draw, shape=circle, fill=white] (v1) at (3,1) {$z$};
\node [draw, shape=circle, fill=white] (v2) at (4,1) {};
\node [draw, shape=circle, fill=white] (v3) at (5,1) {};
\node [draw, shape=circle, fill=white] (x) at (6,1) {$x$};
\node [draw, shape=circle, fill=white] (v4) at (3,-1) {};
\node [draw, shape=circle, fill=white] (v5) at (3,0) {};
\node [draw, shape=circle, fill=white] (v6) at (4,0) {};
\node [draw, shape=circle, fill=white] (v7) at (5,0) {};
\draw (wi) -- (wj);
\draw (wi) -- (v1);
\draw (wi) -- (v4);
\draw (wi) -- (v5);
\draw (v1) -- (v2);
\draw (v2) -- (v3);
\draw (v3) -- (x);
\draw (v5) -- (v6);
\draw (v6) -- (v7);
\end{tikzpicture}
\caption{In this example, $w$ distinguishes the pair $x,y$, and $z$ is the only vertex in $V_{w}$ that does not distinguish $x,y$.} \label{counterexample}
\end{figure}
{\bf Case 2:} $x\in V'$ or $y\in V'$. Thus, $V'\ne \emptyset$ and, as a consequence, $\vert\mathcal{M}(T)\vert \ge 2$. Hence, we have one of the following situations.
\begin{itemize}
\item $\ $There exist two vertices $w,w'\in \mathcal{M}(T)$, $w\ne w'$, such that the shortest path from $x$ to $w$ and the shortest path from $y$ to $w'$ have empty intersection, or
\item $\ $for every vertex $w''\in \mathcal{M}(T)$, it follows that either $y$ belongs to the shortest path from $x$ to $w''$  or $x$ belongs to the shortest  path from $y$ to $w''$.
\end{itemize}
In the first case, $x,y$ are distinguished by vertices in $V_w$ or by vertices in  $V_{w'}$ and in the second one, $x,y$ are distinguished by vertices in $V_{w''}$.

{\bf Case 3:} $x,y\in V_{w}$ for some $w\in \mathcal{M}(T)$. If $x,y\in V(P(u_{l},w))$ for some $l\in \{1,\ldots,\ter(w)\}$, then there exists at most one vertex of $V(P(u_{l},w))$ which does not distinguish $x,y$. Since $\ter(w)\ge 2$, the vertex  $w$ has a terminal vertex $u_{q}$ with $q\ne l$. So, $x,y$ are distinguished by at least $|V(P(u_{l},w,u_{q}))|-1$ vertices, and since $|V(P(u_{l},w,u_{q}))|\ge \varsigma(T)+1$, we are done. If $x\in V(P(u_{l},w)$ and $y\in V(P(u_{q},w)$ for some $l,q\in \{1,\ldots,\ter(w)\}$, $l\ne q$, then there exists at most one vertex of $V(P(u_{l},w,u_{q}))$ which does not distinguish $x,y$. Since $|V(P(u_{l},w,u_{q}))|\ge \varsigma(T)+1$, the result follows.

Therefore,  $\varsigma(T)=\displaystyle\min_{x,y\in V(T)}\vert {\cal D}_T(x,y)\vert$ and by Theorem \ref{theokmetric}  the result follows.
\end{proof}

Since any path is a particular case of a tree and its behavior with respect to the $k$-metric dimension is relative different, here we analyze them in first instance. In Proposition \ref{theoPath2} we noticed that the $2$-metric dimension of a path $P_{n}(n\geqslant 2)$ is two. Here we give a formula for the $k$-metric dimension of any path graph for $k\ge 3$.

\begin{proposition}\label{propPath}
$\ $Let $k\geq 3$ be an integer. For any path graph $P_{n}$ of order $n\geq k+1$, $$\dim_{k}(P_{n})=k+1.$$
\end{proposition}

\begin{proof}
$\ $Let $v_{1}$ and $v_{n}$ be the leaves of $P_{n}$ and let $S$ be a $k$-metric basis of $P_n$.  Since $|S|\ge k\ge 3$, there exists at least one vertex $w\in S\cap (V(P_{n})-\{v_1,v_n\})$. For any vertex $w\in V(P_{n})-\{v_1,v_n\}$  there exist at least two vertices $u,v\in V(P_{n})$ such that $w$ does not distinguish $u$ and $v$. Hence, $|S|=\dim_{k}(P_{n})\geq k+1$.

Now, notice that for any pair of different vertices $u,v\in V(P_{n})$ there exists at most one vertex $w\in V(P_{n})-\{v_{1},v_{n}\}$ such that $w$ does not distinguish $u$ and $v$. Thus, we have that for every $S\subseteq V(P_{n})$ such that $|S|=k+1$ and every pair of different vertices $x,y\in V(P_{n})$, there exists at least $k$ vertices of $S$ such that they distinguish $x,y$. So $S$ is a $k$-metric generator for $P_n$. Therefore, $\dim_{k}(P_{n})\le |S|=k+1$ and, consequently, the result follows.
\end{proof}

Once studied the path graphs, we are now able to give a formula for the $r$-metric dimension of any $k$-metric dimensional tree different from a path which, among other usefulness, shows that Theorem \ref{theoMuk} is tight.

\begin{theorem}\label{theoTreeDimR}
$\ $If $T$ is a tree which is not a path, then for any $r\in \{1,\ldots, \varsigma(T)\}$, $$\dim_{r}(T)=\sum_{w\in \mathcal{M}(T)}I_{r}(w).$$
\end{theorem}

\begin{proof}
$\ $Since $T$ is not a path, $T$ contains at least one vertex belonging to $\mathcal{M}(T)$. Let $w\in \mathcal{M}(T)$ and let $T_{w}=(V_{w},E_{w})$ be  the $w$-branch. Also we consider the set  $V'=V(T)-\bigcup_{w\in \mathcal{M}(T)} V_{w}$.  For every $w\in \mathcal{M}(T)$, we suppose $u_{1}$ is a terminal vertex of $w$ such that $l(u_{1},w)=l(w)$. Let $U(w)=\{u_1,u_{2},\ldots,u_{s}\}$ be the set of terminal vertices of $w$. Now, for every $u_{j}\in U(w)$, let the path $P(u_{j},w)=u_{j}u^1_{j}u^2_{j}\ldots u^{l(u_{j},w)-1}_{j}w$ and we consider the set $S(u_{j},w)\subset V\left(P(u_{j},w)\right)-\{w\}$ given by:
\[
S(u_{1},w)=\left\{
\begin{array}{ll}

\left\{u_{1},u^1_{1},\ldots,u^{l(w)-1}_{1}\right\}, & \mbox{if } l(w)\le\lfloor\frac{r}{2}\rfloor 
\\
\\
\left\{u_{1},u^1_{1},\ldots,u^{\lfloor\frac{r}{2}\rfloor-1}_{1}\right\}, & \mbox{if } l(w)>\lfloor\frac{r}{2}\rfloor . 
\end{array}
\right.
\]
and for $j\ne 1$,
\[
S(u_{j},w)=\left\{
\begin{array}{ll}
\left\{u_{j},u^1_{j},\ldots,u^{r-l(w)-1}_{j}\right\}, & \mbox{if } l(w)\le\lfloor\frac{r}{2}\rfloor ,\\
\\
\left\{u_{j},u^1_{j},\ldots,u^{\lceil\frac{r}{2}\rceil-1}_{j}\right\}, & \mbox{if } l(w)>\lfloor\frac{r}{2}\rfloor .
\end{array}
\right.
\]
According to this we have,
\[
\left|S(u_{j},w)\right|=\left\{
\begin{array}{ll}
\vspace{0.2cm}
l(w), & \mbox{if } l(w)\le\lfloor\frac{r}{2}\rfloor \mbox{ and } u_{j}=u_{1},\\
\vspace{0.2cm}
r-l(w), & \mbox{if } l(w)\le\lfloor\frac{r}{2}\rfloor \mbox{ and } u_{j}\ne u_{1},\\
\vspace{0.2cm}
\lfloor\frac{r}{2}\rfloor, & \mbox{if } l(w)>\lfloor\frac{r}{2}\rfloor \mbox{ and } u_{j}=u_{1},\\
\vspace{0.2cm}
\lceil\frac{r}{2}\rceil, & \mbox{if } l(w)>\lfloor\frac{r}{2}\rfloor \mbox{ and } u_{j}\ne u_{1}.
\end{array}
\right.
\]
Let $S(w)=\displaystyle\bigcup_{u_{j}\in U(w)}S(u_{j},w)$ and $S=\displaystyle\bigcup_{w\in \mathcal{M}(T)}S(w)$. Since for every $w\in \mathcal{M}(T)$ it follows that $\displaystyle\bigcap_{u_{j}\in U(w)}S(u_{j},w)=\emptyset$ and $\displaystyle\bigcap_{w\in \mathcal{M}(T)}S(w)=\emptyset$, we obtain that $|S|=\displaystyle\sum_{w\in \mathcal{M}(T)}I_{r}(w)$.

Also notice that for every $w\in \mathcal{M}(T)$, such that $\ter(w)=2$ we have $|S(w)|= r$ and, if $\ter(w)>2$, then we have $|S(w)|\ge r+1$. We claim that $S$ is an $r$-metric generator for $T$. Let $u,v$ be two distinct vertices of $T$. We consider the following cases.

{\bf Case 1:} $u,v\in V_{w}$ for some $w\in \mathcal{M}(T)$. We have the following subcases.

{\bf Subcase 1.1:} $u,v\in V(P(u_{j},w))$ for some $j\in \{1,\ldots,\ter(w)\}$. Hence there exists at most one vertex of $S(w)\cap V(P(u_{j},w))$ which does not distinguish $u,v$. If $\ter(w)=2$, then there exists at least one more exterior major vertex $w'\in \mathcal{M}(T)-\{w\}$. So, the elements of $S(w')$  distinguish $u,v$. Since $|S(w')|\ge r$, we deduce that at least $r$ elements of $S$ distinguish $u,v$. On the other hand, if $\ter(w)>2$, then since $|S(w)|\ge r+1$, we obtain that at least $r$ elements of  $S(w)$ distinguish $u,v$.

{\bf Subcase 1.2:} $u\in V(P(u_{j},w))$ and $v\in V(P(u_{l},w))$ for some $j,l\in \{1,\ldots,\ter(w)\}$, $j\ne l$. According to the construction of the set $S(w)$, there exists at most one vertex of ($S(w)\cap (V(P(u_{j},w,u_{l}))$) which does not distinguish $u,v$.

Now,  if $\ter(w)=2$, then there exists  $w'\in \mathcal{M}(T)-\{w\}$. If $d(u,w)=d(v,w)$, then  the $r$ elements of $S(w)$ distinguish $u,v$ and, if $d(u,w)\ne d(v,w)$, then the elements of $S(w')$ distinguish $u, v$.

On the other hand, if $\ter(w)>2$, then since $|S(w)|\ge r+1$, we deduce that at least $r$ elements of  $S(w)$ distinguish $u,v$.

{\bf Case 2:} $u\in V_{w}, v\in V_{w'}$, for some $w,w'\in \mathcal{M}(T)$ with $w\ne w'$. In this case, either the vertices in $S(w)$ or the vertices in $S(w')$ distinguish $u,v$. Since  $|S(w)|\ge r$ and $|S(w')|\ge r$ we have that $u,v$ are distinguished by at least $r$ elements of $S$.

{\bf Case 3:} $u\in V'$ or $v\in V'$. Without loss of generality we assume $u\in V'$. Since $V'\ne \emptyset$, we have that there exist at least two different vertices in $\mathcal{M}(T)$. Hence, we have either one of the following situations.
\begin{itemize}
\item $\ $There exist two vertices $w,w'\in \mathcal{M}(T)$, $w\ne w'$, such that the shortest path from $u$ to $w$ and the shortest path from $v$ to $w'$ have empty intersection, or
\item $\ $for every vertex $w''\in \mathcal{M}(T)$, it follows that either $v$ belongs to every shortest  path from $u$ to $w''$ or $u$ belongs to every shortest  path from $v$ to $w''$.
\end{itemize}
Notice that in both situations, since $|S(w)|\ge r$,  for every $w\in\mathcal{M}(T)$), we have that $u,v$ are distinguished by  at least $r$ elements of $S$. In the first case, $u$ and $v$ are distinguished by the elements of $S(w)$ or by the elements of $S(w')$  and, in the second one, $u$ and $v$ are distinguished by  the elements of  $S(w'')$.

Therefore, $S$ is an $r$-metric generator for $T$ and, by Theorem \ref{theoMuk}, the proof is complete.
\end{proof}

In the case $r=1$, the formula of Theorem \ref{theoTreeDimR} leads to 
$$\dim(T)=\mu(T)-|\mathcal{M}(T)|,$$
which is a result obtained in \cite{Chartrand2000}. Other interesting particular cases are the following ones for $r=2$ and $r=3$, respectively. That is, by Theorem \ref{theoTreeDimR} we have the next results.

\begin{corollary}\label{corotree2}
$\ $If $T$ is a tree different from a path, then $$\dim_{2}(T)=\mu(T).$$
\end{corollary}

\begin{corollary}\label{corotree3}
$\ $If $T$ is a tree different from a path with $\varsigma(T)\ge 3$, then $$\dim_{3}(T)=2\mu(T)-|\mathcal{M}(T)|.$$
\end{corollary}

As mentioned before, the two corollaries above show that the bounds given in Corollaries \ref{coroMu2} and \ref{coroMu3} are achieved. We finish our exposition with a formula for the $k$-metric dimension of a $k$-metric dimensional tree with some specific structure, also showing that the inequality $\dim_{k}(T)\ge \vert{\cal D}_{k}(T)\vert$, given in Remark \ref{remTauk}, can be reached.

\begin{proposition}\label{propTauk}
$\ $Let $T$ be a tree different from a path and let $k\ge 2$ be an integer. If $\ter(w)=2$ and $\varsigma(w)=k$ for every  $w\in \mathcal{M}(T)$,   then $\dim_{k}(T)=\vert{\cal D}_{k}(T)\vert$.
\end{proposition}

\begin{proof}
$\ $Since every vertex $w\in \mathcal{M}(T)$ satisfies that $\ter(w)=2$ and $\varsigma(w)=k$, we have that $\varsigma(T)=k$. Thus, by Theorem \ref{theoTreeDimK}, $T$ is $k$-metric dimensional tree. Since $I_k(w)=k$ for every $w\in \mathcal{M}(T)$, by Theorem \ref{theoTreeDimR} we have that $\dim_{k}(T)=k|\mathcal{M}(T)|$. Let $u_{r}, u_{s}$ be the terminal vertices of $w$. As we have shown in the proof of Theorem \ref{theoTreeDimK}, for every pair $x,y\in V(T)$ such that $x\notin V\left(P(u_{r},w,u_{s})\right)-\{w\}$ or $y\notin V\left(P(u_{r},w,u_{s})\right)-\{w\}$, it follows that  $x,y$ are distinguished by at least   $k+1$ vertices of $T$   and so   $\vert{\cal D}_T^*(x,y)\vert > k-2$. Hence, if $\vert{\cal D}_T^*(x,y)\vert = k-2$, then $x,y\in V\left(P(u_{r},w,u_{s})\right)-\{w\}$ for some $w\in \mathcal{M}(T)$. If $d(x,w)\ne d(y,w)$, then $x,y$ are distinguished by more than $k$ vertices (those vertices not in $V\left(P(u_{r},w,u_{s})\right)-\{w\}$). Thus, if $\vert{\cal D}_T^*(x,y)\vert = k-2$, then $d(x,w)=d(y,w)$ and, as a consequence,  ${\cal D}^*_T(x,y)=V\left(P(u_{r},w,u_{s})\right)-\{x,y,w\}$. Considering that $\left|V\left(P(u_{r},w,u_{s})\right)-\{w\}\right|=k$ and at the same time  that $\displaystyle\bigcap_{w\in \mathcal{M}(T)}V\left(P(u_{r},w,u_{s})\right)=\emptyset$, we deduce $\vert{\cal D}_{k}(T)\vert=k|\mathcal{M}(T)|$. Therefore, $\dim_{k}(T)=\vert{\cal D}_{k}(T)\vert$.
\end{proof}

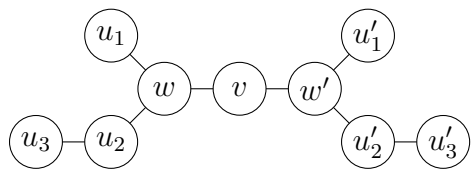
\begin{figure}[h]
\centering
\begin{tikzpicture}[transform shape, inner sep = 1pt, outer sep = 0pt, minimum size = 20pt]
\node [draw=black, shape=circle, fill=white,text=black] (w1) at (-1,0) {$w$};
\node [draw=black, shape=circle, fill=white,text=black] (u11) at ([shift=({135:1})]w1) {$u_{1}$};
\node [draw=black, shape=circle, fill=white,text=black] (u12) at ([shift=({225:1})]w1) {$u_{2}$};
\node [draw=black, shape=circle, fill=white,text=black] (u13) at ([shift=({-1,0})]u12) {$u_{3}$};
\node [draw=black, shape=circle, fill=white,text=black] (w2) at (1,0) {$w'$};
\node [draw=black, shape=circle, fill=white,text=black] (u21) at ([shift=({45:1})]w2) {$u_{1}'$};
\node [draw=black, shape=circle, fill=white,text=black] (u22) at ([shift=({315:1})]w2) {$u_{2}'$};
\node [draw=black, shape=circle, fill=white,text=black] (u23) at ([shift=({1,0})]u22) {$u_{3}'$};
\node [draw=black, shape=circle, fill=white,text=black] (v) at (0,0) {$v$};
\draw[black] (u13) -- (u12)  -- (w1) -- (v) -- (w2) -- (u22) -- (u23);
\draw[black] (w1) -- (u11);
\draw[black] (w2) -- (u21);
\end{tikzpicture}
\caption{A $3$-metric dimensional tree $T$ for which $\dim_{3}(T)=\vert{\cal D}_3(T)\vert=6$.}
\label{figTree3D}
\end{figure}

Figure \ref{figTree3D} shows an  example of a $3$-metric dimensional tree.   In this case $\mathcal{M}(T)=\{w,w'\}$, $\ter(w)=\ter(w')=2$ and $\varsigma(w)=\varsigma(w')=3$. Then Proposition \ref{propTauk} leads to $\dim_{3}(T)=\vert{\cal D}_3(T)\vert=|\{u_{1},u_{2},u_{3},u_{1}',u_{2}',u_{3}'\}|=6$.


\begin{thebibliography}{9}
\expandafter\ifx\csname url\endcsname\relax
  \def\url#1{\texttt{#1}}\fi
\expandafter\ifx\csname urlprefix\endcsname\relax\def\urlprefix{URL }\fi

\bibitem{Bailey2011a}
R.~F. Bailey, P.~J. Cameron, Base size, metric dimension and other invariants
  of groups and graphs, Bulletin of the London Mathematical Society 43~(2)
  (2011) 209--242.
\newline\urlprefix\url{http://blms.oxfordjournals.org/content/43/2/209.1.abstract}

\bibitem{Brigham2003}
R.~C. Brigham, G.~Chartrand, R.~D. Dutton, P.~Zhang, Resolving domination in
  graphs, Mathematica Bohemica 128~(1) (2003) 25--36.
\newline\urlprefix\url{http://mb.math.cas.cz/mb128-1/3.html}

\bibitem{Caceres2007}
J.~C\'{a}ceres, C.~Hernando, M.~Mora, I.~M. Pelayo, M.~L. Puertas, C.~Seara,
  D.~R. Wood, On the metric dimension of cartesian product of graphs, SIAM
  Journal on Discrete Mathematics 21~(2) (2007) 423--441.
\newline\urlprefix\url{http://epubs.siam.org/doi/abs/10.1137/050641867}

\bibitem{Chappell2008}
G.~G. Chappell, J.~Gimbel, C.~Hartman, Bounds on the metric and partition
  dimensions of a graph, Ars Combinatoria 88 (2008) 349--366.
\newline\urlprefix\url{http://www.cs.uaf.edu/~chappell/papers/metric/metric.pdf}

\bibitem{Chartrand2000}
G.~Chartrand, L.~Eroh, M.~A. Johnson, O.~R. Oellermann, Resolvability in graphs
  and the metric dimension of a graph, Discrete Applied Mathematics 105~(1-3)
  (2000) 99--113.
\newline\urlprefix\url{http://dx.doi.org/10.1016/S0166-218X(00)00198-0}

\bibitem{Chartrand2000a}
G.~Chartrand, C.~Poisson, P.~Zhang, Resolvability and the upper dimension of
  graphs, Computers \& Mathematics with Applications 39~(12) (2000) 19--28.
\newline\urlprefix\url{http://dx.doi.org/10.1016/S0898-1221(00)00126-7}

\bibitem{Chartrand2000b}
G.~Chartrand, E.~Salehi, P.~Zhang, The partition dimension of a graph,
  Aequationes Mathematicae 59~(1-2) (2000) 45--54.
\newline\urlprefix\url{http://dx.doi.org/10.1007/PL00000127}

\bibitem{Fehr2006}
M.~Fehr, S.~Gosselin, O.~R. Oellermann, The partition dimension of cayley
  digraphs, Aequationes Mathematicae 71~(1-2) (2006) 1--18.
\newline\urlprefix\url{http://link.springer.com/article/10.1007%2Fs00010-005-2800-z}

\bibitem{Harary1976}
F.~Harary, R.~A. Melter, On the metric dimension of a graph, Ars Combinatoria 2
  (1976) 191--195.
\newline\urlprefix\url{http://www.ams.org/mathscinet-getitem?mr=0457289}

\bibitem{Haynes2006}
T.~W. Haynes, M.~A. Henning, J.~Howard, Locating and total dominating sets in
  trees, Discrete Applied Mathematics 154~(8) (2006) 1293--1300.
\newline\urlprefix\url{http://www.sciencedirect.com/science/article/pii/S0166218X06000035}

\bibitem{Johnson1993}
M.~Johnson, Structure-activity maps for visualizing the graph variables arising
  in drug design, Journal of Biopharmaceutical Statistics 3~(2) (1993)
  203--236, pMID: 8220404.
\newline\urlprefix\url{http://www.tandfonline.com/doi/abs/10.1080/10543409308835060}

\bibitem{Johnson1998}
M.~A. Johnson, Browsable structure-activity datasets, in: R.~Carb\'{o}-Dorca,
  P.~Mezey (eds.), Advances in Molecular Similarity, chap.~8, JAI Press Inc,
  Stamford, Connecticut, 1998, pp. 153--170.
\newline\urlprefix\url{http://books.google.es/books?id=1vvMsHXd2AsC}

\bibitem{Khuller1996}
S.~Khuller, B.~Raghavachari, A.~Rosenfeld, Landmarks in graphs, Discrete
  Applied Mathematics 70~(3) (1996) 217--229.
\newline\urlprefix\url{http://www.sciencedirect.com/science/article/pii/0166218X95001062}

\bibitem{Kuziak2013}
D.~Kuziak, I.~G. Yero, J.~A. Rodr\'{\i}guez-Vel\'{a}zquez, On the strong metric
  dimension of corona product graphs and join graphs, Discrete Applied
  Mathematics 161~(7--8) (2013) 1022--1027.
\newline\urlprefix\url{http://www.sciencedirect.com/science/article/pii/S0166218X12003897}

\bibitem{Oellermann2007}
O.~R. Oellermann, J.~Peters-Fransen, The strong metric dimension of graphs and
  digraphs, Discrete Applied Mathematics 155~(3) (2007) 356--364.
\newline\urlprefix\url{http://www.sciencedirect.com/science/article/pii/S0166218X06003015}

\bibitem{Okamoto2010}
F.~Okamoto, B.~Phinezy, P.~Zhang, The local metric dimension of a graph,
  Mathematica Bohemica 135~(3) (2010) 239--255.
\newline\urlprefix\url{http://dml.cz/dmlcz/140702}

\bibitem{Saenpholphat2004}
V.~Saenpholphat, P.~Zhang, Conditional resolvability in graphs: a survey,
  International Journal of Mathematics and Mathematical Sciences 2004~(38)
  (2004) 1997--2017.
\newline\urlprefix\url{http://www.hindawi.com/journals/ijmms/2004/247096/abs/}

\bibitem{Sebo2004}
A.~Seb\"{o}, E.~Tannier, On metric generators of graphs, Mathematics of
  Operations Research 29~(2) (2004) 383--393.
\newline\urlprefix\url{http://dx.doi.org/10.1287/moor.1030.0070}

\bibitem{Slater1975}
P.~J. Slater, Leaves of trees, Congressus Numerantium 14 (1975) 549--559.

\bibitem{Slater1988}
P.~J. Slater, Dominating and reference sets in a graph, Journal of Mathematical
  and Physical Sciences 22~(4) (1988) 445--455.
\newline\urlprefix\url{http://www.ams.org/mathscinet-getitem?mr=0966610}

\bibitem{Tomescu2008}
I.~Tomescu, Discrepancies between metric dimension and partition dimension of a
  connected graph, Discrete Applied Mathematics 308~(22) (2008) 5026--5031.
\newline\urlprefix\url{http://www.sciencedirect.com/science/article/pii/S0012365X07007200}

\bibitem{Yero2011}
I.~G. Yero, D.~Kuziak, J.~A. Rodr\'{\i}quez-Vel\'{a}zquez, On the metric
  dimension of corona product graphs, Computers \& Mathematics with
  Applications 61~(9) (2011) 2793--2798.
\newline\urlprefix\url{http://www.sciencedirect.com/science/article/pii/S0898122111002094}

\bibitem{Yero2010}
I.~G. Yero, J.~A. Rodr\'{\i}quez-Vel\'{a}zquez, A note on the partition
  dimension of cartesian product graphs, Applied Mathematics and Computation
  217~(7) (2010) 3571--3574.
\newline\urlprefix\url{http://www.sciencedirect.com/science/article/pii/S0096300310008921}

\end{thebibliography}
\end{document}